\documentclass[ctagsplt,12pt]{amsart}
\RequirePackage{newlfont}
\usepackage{amscd}
\usepackage{amsmath}
\usepackage{amsfonts}
\usepackage{amssymb}
\usepackage{enumerate}
\usepackage{graphicx}
\usepackage{latexsym}
\usepackage{color}
\usepackage{bbm}

\usepackage{relsize} 

\numberwithin{equation}{section}

\begin{document}

	\title[Limiting behavior]{Limit as $p(x)\rightarrow \infty$ of $p(x)$-Harmonic functions for unbounded $p(x)$}
	
	\author[ B. Djafari Rouhani, J. Lang, O. M\'{e}ndez]{ B. Djafari Rouhani, J. Lang, O.M\'{e}ndez}

	\address{Osvaldo M\'{e}ndez\\Department of Mathematical Sciences, The University of Texas at El Paso, El Paso, TX 79968, USA}
	\email{osmendez@utep.edu}
	\address{Behzad Rouhani\\ Department of Mathematical Sciences, University of Texas at El Paso}
	\email{behzad@math.utep.edu}
	
	\subjclass[2020]{Primary 46A16, Secondary 46B20, 46E30, 46A20}
	\keywords{Luxemburg norm, modular topology, modular vector space, variable exponent spaces, unbounded exponents.}
	
	\begin{abstract}
		It is shown that if $p_n$ is a sequence of continuous, unbounded exponents on a bounded, smooth domain $\Omega\subset {\mathbb R}^n$ with $1<\inf\limits_{x\in \Omega}p_n(x)$ and $p_n\rightarrow \infty$ uniformly, then the sequence $(u_n)$ of solutions of the $p_n(\cdot)$-Laplacian converges to the viscosity solution of a suitable differential operator. The novelty here is that each term of the sequence of exponents $(p_n)$ is allowed to be unbounded in $\Omega$. 
	\end{abstract}
	\maketitle
	\newtheorem{question}{Question}
	\newtheorem{property}{Property}
	\newtheorem{theorem}{Theorem}[section]
	\newtheorem{acknowledgement}{Acknowledgement}
	\newtheorem{algorithm}{Algorithm}
	\newtheorem{axiom}{Axiom}[section]
	\newtheorem{case}{Case}
	\newtheorem{claim}{Claim}
	\newtheorem{conclusion}{Conclusion}
	\newtheorem{condition}{Condition}
	\newtheorem{conjecture}{Conjecture}
	\newtheorem{corollary}{Corollary}[section]
	\newtheorem{criterion}{Criterion}
	\newtheorem{definition}{Definition}[section]
	\newtheorem{example}{Example}[section]
	\newtheorem{exercise}{Exercise}
	\newtheorem{lemma}{Lemma}[section]
	\newtheorem{notation}{Notation}
	\newtheorem{problem}{Problem}
	\newtheorem{proposition}{Proposition}[section]
	\newtheorem{remark}{Remark}[section]
	\newtheorem{solution}{Solution}
	\newtheorem{summary}{Summary}
	\newtheorem{assumption}{Assumptions}

	\thispagestyle{empty}

\section{Introduction}\label{background}

This note complements the line of work initiated in \cite{BO}. Specifically, we study the limiting behavior of the solutions of a sequence of Dirichlet problems involving the $p(\cdot)$-Laplacian operator, in the spirit of \cite{MRU}. The fundamental difference between our work and \cite{MRU} is that we allow for the involved variable exponents to be {\it unbounded} in $\Omega$, that is, we include the case $\sup\limits_{x\in \Omega}p(x)=\infty.$\\ 

Specifically, let $\Omega\subset {\mathbb R}^n$ be a bounded, $C^2$ domain, $p:\Omega \rightarrow {\mathbb R}$ be a continuous function such that $1<p_-:=\inf\limits_{x\in \Omega}\ p(x)\leq p_+:=\sup\limits_{x\in \Omega}\ p(x)$. Let $\varphi\in W^{1,p(\cdot)}(\Omega)$ satisfy $\int_{\Omega}p(x)^{-1}|\nabla \varphi(x)|^{p(x)}dx<\infty$. The Dirichlet problem
\begin{equation}\label{DP}
	\begin{cases}
		\Delta_{p(\cdot)}(w) = \text{div}\left( |\nabla w|^{p(\cdot)-2} \nabla w \right) = 0 & \text{in} \ \Omega, \\
		w|_{\partial \Omega} = \varphi,
	\end{cases}
\end{equation}
calls for a $p(x)$-harmonic function $w\in W^{1,p(\cdot)}(\Omega)$ such that $w-\varphi$ vanishes on the boundary $\partial\Omega$; this statement will be written shortly as $u|_{\partial \Omega}=\varphi$ and its meaning will be precised later. 

\smallskip

On the other hand, the limiting case of the $p$-Laplacian operator $\Delta_{p(\cdot)}$ is well understood and takes up the form of the so called infinity Laplacian, namely		
\begin{equation}	\label{infinitylaplacian}
	\Delta_{\infty}u=\langle D^2(u),\nabla u\rangle\cdot\nabla u=\sum\limits_{i,j=1}^n u_{x_j}u_{x_i}u_{x_ix_j}.
\end{equation}
As usual, $\langle A ,a \rangle $ stands for the standard action of the matrix $A$ on the vector $a$ and $\cdot$ denotes the  dot product on ${\mathbb R}^n$. See \cite{MRU} and the references therein for further details regarding this operator.\\
Our main contribution is Theorem \ref{theoremain}, which shows that letting $p(x)\rightarrow \infty$ uniformly in $\Omega$, the solution $u_p$ of problem (\ref{DP}) converge uniformly to the unique viscosity solution of 
\begin{equation}	\label{dirichletinfinitylaplacian}
	\begin{cases}	\Delta_{\infty}u=0\,\,\text{in}\,\,\Omega \\
		u|_{\partial\Omega}=\varphi.
	\end{cases}
\end{equation}
Theorem \ref{theoremain} improves the results in \cite{MRU}, since it does not rely on the implicit assumption $\sup\limits_{x\in \Omega}p(x)<\infty$, which is used in the proof of. \cite[Theorem 1.1]{MRU} in a fundamental way.\\
The variability of the exponent $p(x)$ adds at least two levels of difficulty not visible in the constant case. Firstly, even for $p(x)<\infty$ a.e. in $\Omega$, it might hold $\sup\limits_{x\in \Omega}p(x)=\infty$. Secondly, it is known that smooth functions with compact support are in general not dense in the space of compactly supported functions in $W^{1,p(\cdot)}(\Omega)$, even when the variable exponent $p$ is bounded \cite{Harj}. In this context, the idea of weak solution of problem (\ref{DP}) is ambiguous \cite{BDS,Harj}.\\
More specifically, it was proved in \cite[Theorem 7.2]{AOJA} that for variable exponent $p$ with $p_->n$, there exists a unique weak solution $u\in W^{1,p}(\Omega)$ of problem (\ref{DP}) satisfying the identity
\begin{equation}\label{weaksolclass}
	\int\limits_{\Omega}|\nabla u|^{p-2}\nabla u\nabla h\,dx=0
\end{equation}
for any $h\in C^{\infty}_0(\Omega).$
However, it is well known that for a variable exponent $p$ on $\Omega$, the norm closure of $C^{\infty}_0(\Omega)$ in $W^{1,p}(\Omega)$ might be strictly contained in the norm closure of the subspace $W^{1,p}_{com}(\Omega)=\{v\in W^{1,p}(\Omega):\,\, supp\,v\,\,\text{compact}\}$ (see \cite{Harj}). Thus, the weak solution obtained in \cite[Theorem 7.2]{AOJA} might fail to satisfy (\ref{weaksolclass}) for $h\in W^{1,p}_{com}(\Omega)$. \\
In more precise terms, let $V^{1,p}_0(\Omega)$ and $U^{1,p}_0(\Omega)$ be the modular closures in $W^{1,p}(\Omega)$ of $C^{\infty}_0(\Omega)$ and $W^{1,p}_{comp}(\Omega)$, respectively. In \cite[Example 3.9]{Harj} the author displays a domain $\Omega$, a variable exponent $q$ and a function $\psi \in U^{1,q}_0(\Omega)\setminus V^{1,q}_0(\Omega)$. The unique solution $v\in W^{1,q}(\Omega)$ given by 
\cite[Theorem 7.2]{AOJA} with boundary value $\varphi=\psi$ satisfies the equality (\ref{weaksolclass}) (replacing $p$ with $q$) for every $h \in C^{\infty}_0(\Omega)$ but not for each $h\in W^{1,q}_{com}(\Omega)$.\\
Our chief concern before undertaking the analysis of the limiting behaviour of the solutions is thus to clarify the type of weak solution of the problem (\ref{DP}) for which our limiting process is valid. We undertake this task in Section \ref{dirichletintegral}.
\\ The main obstacle in the analysis of Section \ref{dirichletintegral} arises from the lack of boundedness of the exponents $p_j$. This prevents the direct application of standard compactness and regularity results such as those in, for example \cite{FZ}, and requires a refined use of modular techniques.

\smallskip
Our presentation is organized as follows. 
In Section \ref{functionspaces} a brief summary of the basic definitions and modular geometric properties of the variable exponent Lebesgue and Sobolev spaces is introduced. Since our approach is essentially modular in nature, in Section \ref{modulartopologies} we summarize the modular concepts required for the proof of Theorem \ref{Dirichletvarphi} and introduce the spaces $V^{1,p}_0(\Omega)$ and $U^{1,p}_0(\Omega)$ defined as the modular closures of $C^{\infty}_0(\Omega)$ and $W^{1,p}_{comp}(\Omega)$, respectively.\\

In Section \ref{dirichletintegral} we prove Theorem \ref{Dirichletvarphi}, which, under the assumption $p_-=\inf\limits_{x\in \Omega}p(x)>n$, asserts existence and uniqueness of a weak solution $w$ of problem (\ref{DP}) such that $ w\in W^{1,p}(\Omega)$ and that satisfies (\ref{weaksolclass}) for all $h\in W^{1,p}_{com}(\Omega)$.\\ It must be emphasized that
in spite of its similarity with \cite[Theorem 7.2]{AOJA}, the results in Section \ref{dirichletintegral} are new, since they involve the minimization of the Dirichlet integral (\ref{Dirichletintegral}) on a space that is strictly larger than that considered in \cite{AOJA} and the weak solution emerging in this way might be different from the one in Theorem 7.2 in \cite{AOJA}.

\medskip
The theme in Section \ref{viscsol} is that the weak solution given by Theorem \ref{Dirichletvarphi} is in fact a viscosity solution of the problem (\ref{DP}). Our proof of Proposition \ref{weakviscosity} differs from Proposition 2.3 in \cite{MRU} in that it is not clear to the authors why the function $\left(\Phi-u\right)^+$ used in \cite[Proposition 2.3]{MRU} qualifies as a test function for the weak formulation of \cite[Problem $(1.1)_n$]{MRU} (i.e., belongs to the norm closure of $C^{\infty}_0(\Omega)$, see Remark \ref{comparisonvu}) below); our function $\Phi$ introduced in (\ref{Phi}) is by definition in $W^{1,p}_{com}(\Omega)$ and thus it fits the definition of weak solution of problem (\ref{problemr1}), as stated in (\ref{weakformulation}).\\	

The central result of this work is presented in Section \ref{main}. Specifically, let $(p_j)$ be a sequence of functions satisfying the assumptions (\ref{c1})-(\ref{c3}) in Section \ref{main}. In particular, we allow $\sup\limits_{x\in \Omega}p_j(x)=\infty.$ Fix $\varphi\in W^{1,\infty}(\Omega)$ and for each $j$, let $u_j\in U^{1,p}_0(\Omega)$ be the unique solution to the Dirichlet problem for variable exponent $p_j(\cdot)$-Laplacian
\begin{equation}\label{problemr1}
	\begin{cases}
		\Delta_{p_j(x)}u(x):=\text{div}\big(|\nabla u(x)|^{p_j(x)-2}\nabla u(x)\big) =0, & x\in\Omega\\
		u(x)=\varphi(x), & x\in\partial\Omega,
	\end{cases}
\end{equation}
whose existence and uniqueness is proved in Theorem \ref{Dirichletvarphi}. Notice that $u_j$ might be different from the solution in $V^{1,p}_0(\Omega)$ obtained in \cite{AOJA}, since in general $V^{1,p}_0(\Omega)\subsetneq U^{1,p}_0(\Omega)$, even for bounded exponents $p_j$, as observed in \cite{Harj}.\\
Then
\[
u_j\to u \quad \text{uniformly in }\Omega,
\]
where $u$ is the unique viscosity solution of
\begin{equation}\label{limitproblem}
	\begin{cases}
		-\Delta_\infty u-|\nabla u|^2\ln|\nabla u|\langle \xi,\nabla u\rangle=0
		& \text{in }\Omega,\\
		u=\varphi & \text{on }\partial\Omega.
	\end{cases}
\end{equation}
Theorem \ref{theoremain} was proved in \cite[Theorem 1.1]{MRU} under the additional assumption $\sup\limits_{x\in \Omega}p_j(x)=p_+<\infty$, see \cite[inequality (2.9)]{MRU}. It is claimed there that the boundedness of each $u_j$ follows from the condition given in \cite[condition (1.4)]{MRU}, namely
\begin{equation}\label{1.4}
	\nabla \ln{p_j}\rightarrow \xi \in C(\Omega)\,\,\text{uniformly in} \,\,\Omega.
\end{equation}
However, (\ref{1.4}) does not guarantee the boundedness of each $p_j$, as evidenced by the example $\Omega=(0,1)$,  $p_j(x)=j\frac{e^{\frac{1}{x}}}{1-x}$, $\xi(x)=-\frac{1}{x^2}+\frac{1}{1-x}$. The results in \cite{MRU} require the more restrictive condition
\begin{equation}
	\nabla \ln{p_j}\rightarrow \xi \in C(\overline{\Omega}).
\end{equation}
\section{Preliminaries}\label{functionspaces}
The main results and definitions in this Section have been extensively discussed in \cite{DHHR,OAP,AOJA,KR}. In the sequel, $\Omega\subset {\mathbb R}^n$ will stand for a bounded, smooth domain with boundary $\partial\Omega$ and $p:\Omega\rightarrow (1,\infty)$ denotes a Borel-measurable function subject to the constraints
\begin{equation}\label{constraintforp}
	1<p_-=\inf\limits_{\Omega}p(x)\leq \sup\limits_{\Omega}p(x)=p_+=\infty.
\end{equation}
For the sake of typographical simplicity, variable exponents will be denoted without specific reference to the spatial variable they depend upon, i.e., $p(x)$ will be written as $p$.\\
The next definition concerns the variable exponent Lebesgue and Sobolev spaces and their corresponding Luxemburg norm.
\begin{definition}\cite{DHHR, KR}\label{deflp}
	In the notation of Section \ref{background}, let $p:\Omega\rightarrow  [1,\infty)$ be Borel measurable. Notice that we include the case $p_+=\sup\limits_{x\in \Omega}p(x)=\infty.$ We set
	\begin{equation*}
		L^{p}(\Omega)=\left\{f:\int\limits_{\Omega}|\lambda f(x)|^{p}\,dx<\infty,\,\text{for some}\,\,\lambda>0\right\}.
	\end{equation*}
	It is well known that when endowed with the Luxemburg norm defined by
	\begin{equation*}
		\|f\|_{p}=\inf\left\{\lambda>0: \int\limits_{\Omega}\left(|f(x)|/\lambda)\right)^{p}dx\leq 1\right\}
	\end{equation*}
	$L^p(\Omega)$ becomes a Banach space, which is uniformly convex if and only if $1<p_-=\inf\limits_{x\in \Omega}p(x)\leq p_+=\sup\limits_{x\in \Omega}p(x)<\infty.$ In particular, if $p_+=\infty$, the Luxemburg norm $\|\cdot\|_{p}$ on $L^{p}(\Omega)$ is not uniformly convex, \cite{Luk}. 
\end{definition}

It is a routine exercise to show that when $p$ is constant on $\Omega$, the above defined spaces coincide with the usual Lebesgue spaces. \\
If $p\leq q$ are measurable in $\Omega$, the embedding
$L^q(\Omega)\hookrightarrow L^p(\Omega)$ is continuous, i.e., there exists a positive constant $C(p,q,\Omega)$ such that
$\|u\|_p\leq C(p,q,\Omega)\|u\|_q$ for any $u\in L^q(\Omega)$ \cite{KR}.\\

\begin{definition}\cite{DHHR,KR}
	\begin{equation*}
		W^{1,p(\cdot)}(\Omega)=\left\{f:f\in L^{p(\cdot)}(\Omega)\,\text{and}\,\, |\nabla f|\in L^{p(\cdot)}(\Omega)\right\},
	\end{equation*}
	where $|\nabla f|$ stands for the Euclidean norm of\,\,$\nabla f$ and the Sobolev norm is defined as
	\begin{equation}\label{sobolevnorm}
		\|f\|_{1,p(\cdot)}=\|f\|_{p(\cdot)}+\||\nabla f|\|_{p(\cdot)}.
	\end{equation}
\end{definition}
\subsection{Modulars and their geometry}\label{modulars}
In this Section we briefly summarize some facts on the theory of modulars that are essential for the rest of the presentation. We refer the reader to \cite{KK, AOJA, M:1983} for a detailed treatment of the ideas merely sketched in this Section.
\begin{definition}\label{def-modular}\cite{M:1983, nakano}
	A convex modular on a real vector space $X$ is a function $\varrho: X \to [0,\infty]$ that satisfies the following conditions:
	\begin{enumerate}
		\item[(1)] $\varrho(x) = 0$ if and only if $x = 0$;
		\item[(2)] $\varrho(\alpha x) = \varrho(x)$, if $|\alpha| =1$;
		\item[(3)] $\varrho(\alpha x + (1-\alpha) y )\leq \alpha\varrho(x) + (1-\alpha)\varrho(y)$, for any $\alpha \in [0,1]$
		and any $x, y \in X$.
	\end{enumerate}
	Moreover, $\varrho$ is said to be left-continuous if, for all $x \in X$,
	$$\lim\limits_{r \to 1^{-}}\ \varrho(rx) = \varrho(x).$$
\end{definition}

\begin{remark}  {\normalfont If the condition (1) is replaced with
		$$\varrho(0) = 0,$$ $\varrho$ is said to be a pseudo-modular (see\cite{M:1983}).}
\end{remark}
Finally, a pseudomodular $\rho$ on a vector space $X$ is said to be uniformly convex ($UC$) if for every $\varepsilon>0$ there exists $\delta=\delta(\varepsilon)>0$ (i.e, $\delta$ is independent of $u,v$) such that  for every $u\in X$ and $v\in X$:
\begin{equation}\label{defunifcom}\rho\left(\frac{u-v}{2}\right) \geq \varepsilon\ \frac{\rho(u)+\rho(v)}{2} \;\; \text{implies}\;\;\;
	\rho\left(\frac{u+v}{2}\right) \leq (1-\delta)\ \frac{\rho(u)+\rho(v)}{2}.
\end{equation}

Theorem \ref{UC-sobolev2} addresses the uniform convexity statement needed in the sequel. Notice that the result is valid even in the case the variable exponent $p(x)$ is unbounded in $\Omega$, that is even if $p_+=\infty.$ We refer the interested reader to \cite{AOJA} for the details of the proof.

\begin{theorem}
	\label{UC-sobolev2} Let $\Omega\subseteq {\mathbb R}^n$ be a domain, $p:\Omega\rightarrow [1,\infty)$ (here for $u\in W^{1,p}(\Omega), u_j=\frac{\partial u}{\partial x_j}$ and $|\cdot|$ stands for the Euclidean norm in ${\mathbb R}^n$).  Consider the functional $\varrho :W^{1,p}(\Omega)\rightarrow [0,\infty]$ defined by
	$$\varrho(u) := \rho_{p}(|\nabla u|)= \int_{\Omega}\frac{1}{p(x)} \left(\sum_1^nu_j^2\right)^{\frac{p(x)}{2}}\ dx = \int_{\Omega}\frac{1}{p(x)}|\nabla u|^{p(x)}\ dx.$$
	Then, if $p_->1$, $\varrho$ is a convex pseudomodular on $W^{1,p}(\Omega)$ and is $(UC)$; in fact, for every $\epsilon>0$ there is $\delta=\delta(\epsilon, p_-)>0$ such that (\ref{defunifcom}) is satisfied. In particular, $\varrho$ is uniformly convex even if $p_+=\infty.$
\end{theorem}
\begin{proof}
	See\cite{AOJA} for the proof.
\end{proof}
\section{Modular topologies and subspaces}\label{modulartopologies}

On the set ${\mathcal M}$ of extended-real-valued Borel-measurable functions defined on $\Omega$, consider the functional $\rho_{p}:{\mathcal M}\rightarrow [0,\infty]$ defined by
\begin{equation}\label{def-rho}
	\rho_{p}(u):=\int_{\Omega}\frac{|u(x)|^{p(x)}}{p(x)}\ dx.
\end{equation}
Then, $\rho_p$ is a left-continuous modular on the vector space
$L^{p}(\Omega)$ and in fact, the following holds:
\begin{lemma}
	Let $\Omega\subseteq {\mathbb R}^n$ and $p:\Omega\rightarrow [1,\infty)$ be measurable. Then
	\begin{align}\nonumber
		L^p(\Omega)&=\{v\in {\mathcal M}: \rho_p(\lambda v)<\infty\,\,\text{for some}\,\lambda>0\}\\ \nonumber &=\{v\in {\mathcal M}: \int\limits_{\Omega}\left|\lambda v\right|^{p}\,dx<\infty\,\,\text{for some}\,\lambda>0\}.
	\end{align}
\end{lemma}
\begin{proof}
	The proof is simple, the reader is referred to \cite{AOJA} for the details.
\end{proof}
\begin{remark}{\normalfont In fact, it can be shown that the Luxemburg norm given in Definition \ref{deflp} is equivalent to the norm 
		\begin{equation*}
			\|f\|_{\ast,p}=\inf\left\{\lambda>0: \int\limits_{\Omega}\frac{1}{p}\left(|f(x)|/\lambda)\right)^{p}dx\leq 1\right\}
		\end{equation*}  
		In other words, the modulars $u\rightarrow \int\limits_{\Omega}|u|^pdx$ and $u\rightarrow \int\limits_{\Omega}\frac{|u|^p}{p}dx$ define the exact same normed space $L^p(\Omega)$ and generate equivalent Luxemburg norms.
		
	}
	
\end{remark}
Likewise, $W^{1,p}(\Omega)$ will stand for the vector subspace of $L^p(\Omega)$ consisting of those functions whose weak derivatives also belong to $L^p(\Omega)$ and it will be endowed with the modular
\begin{align}
	\rho_{1,p}:W^{1,p}(\Omega)\rightarrow [0,\infty]\\ \nonumber
	\rho_{1,p}(u):=\rho_{p}(u)+\rho_p\left(|\nabla u|\right)
\end{align}
The modulars $\rho_p$, $\rho_{1,p}$ introduced above define  Hausdorff topologies in $L^p(\Omega)$ and $W^{1,p}(\Omega)$, respectively. \\
Specifically, a set $O\subseteq W^{1,p}(\Omega)$ is declared to be $\rho_{1,p}$ open iff for any $u\in O$ there is $r>0$ such that the modular ball $\{v\in W^{1,p}(\Omega):\rho_{1,p}(v-u)<r\}$ is contained in $O$. This is equivalent to defining a set $A$ to be closed iff for any sequence $(u_j)\subseteq A$ such that $\rho_{1,p}(u_j-u)\rightarrow 0$, it must hold $u\in A$. The topology induced on $L^p(\Omega)$ by the modular $\rho_p$ is defined analogously.\\	
These topologies will be from now on referred to as the ($\rho_p$) $\rho_{1,p}$ topology, respectively, or the modular topology if there is no room for confusion, and denoted by  ($\tau_{\rho_{p}}$) $\tau_{\rho_{1,p}}$. We refer the reader to \cite{AOJA} for a detailed treatment of the modular topology and simply state the following Theorems that will be used in the sequel. In what follows, the $\tau_{\rho_{1,p}}$-closure of a set $S\subset W^{1,p}(\Omega)$ will be written as $\overline{S}^{\rho_{1,p}}$. It is worth noticing that modular balls are generally {\it not} open in the modular topology.\\
Theorems \ref{basic-properties}-\ref{pqmodularembedding1} were proved in \cite{AOJA} and will be essential in the sequel.
\begin{theorem}\label{basic-properties} The following properties hold:
	\begin{enumerate}
		\item If $U$ is a $\rho_{1,p}$-open subset of $W^{1,p}(\Omega)$, then, for any $v\in W^{1,p}(\Omega)$ the set $U + v = \{u+v; u \in U\}$ is also $\rho_{1,p}$-open.  Hence $U+V$ is $\rho_{1,p}$-open provided either $U$ or $V$ is $\rho_{1,p}$-open.
		\item If $U$ is $\varrho$-open and $\theta \geq 1$, then $\theta U$ is also $\rho_{1,p}$-open.
		\item For any $v \in \overline{A}^{\rho_{1,p}}$ and any $\rho_{1,p}$-open set $U$ such that $v \in U$, then $U\cap A \neq \emptyset$.
		\item If $A$ is convex, then $\overline{A}^{\rho_{1,p}}$ is convex.
	\end{enumerate}
\end{theorem}
\begin{proof}
	See \cite[Section 4]{AOJA}.
\end{proof}
\begin{theorem}\label{subspace}  Let $A$ be a vector subspace of $W^{1,p}(\Omega)$.  Then $\overline{A}^{\varrho}$ is a $\varrho$-closed vector subspace of $W^{1,p}(\Omega)$.
\end{theorem}
\begin{proof}
	See \cite[Prop. 4.10]{AOJA}
\end{proof}
Moreover, the modular topology is complete. We make this statement precise in the following theorems.
\begin{theorem}\label{completeness}
	$L^p(\Omega)$ is $\rho_p$-complete, that is, if $(u_j)\subset L^p(\Omega)$ and $\rho_p(u_j-u_k)\rightarrow 0$ as $j,k\rightarrow \infty$, then there exists $u\in L^p(\Omega)$ such that $\rho_p(u_j-u)\rightarrow 0$ as $j\rightarrow \infty$.
\end{theorem}
\begin{proof}
	See \cite{M:1983}.
\end{proof}
\begin{theorem}\label{completeness1}
	If $(u_j)\subset W^{1,p}(\Omega)$ and $\rho_{1,p}(u_j-u_k)\rightarrow 0$ as $j,k\rightarrow \infty$, then there exists $u\in W^{1,p}(\Omega)$ such that $\rho_{1,p}(u-u_j)\rightarrow 0$ as $j\rightarrow \infty$. Thus, the Sobolev space $W^{1,p}(\Omega)$ is $\rho_{1,p}$-complete. 
\end{theorem}
\begin{proof}
	See \cite[Theorem 4.22]{AOJA}. 
\end{proof}
\begin{theorem}\label{pqmodularembedding1}
	For a bounded domain $\Omega\subset {\mathbb R}^n$ and Borel measurable functions $p,q $ on $\Omega$ with $1\leq q\leq p$ a.e. and such that $\int\limits_{\Omega}\frac{e^{q(x)}}{q(x)}dx<\infty$,  the following statements hold:
	\begin{enumerate}
		\item [(i)] If $(u_j)$ converges to $u$ in $W^{1,p}(\Omega)$, then necessarily $\lim\limits_{j\rightarrow \infty}\rho_{1,q}(u_j - u) = 0$, i.e., $(u_j)$ converges to $u $ in $W^{1,q}(\Omega)$.\\
		\item [(ii)] The inclusions		
		\begin{equation*}
			j_{p,q}: \left(L^{p}(\Omega), \tau_{p}\right)\longrightarrow \left(L^{q}(\Omega), \tau_{q}\right)
		\end{equation*}
		and 
		\begin{equation*}
			i_{p,q}: \left(W^{1,p}(\Omega), \tau_{1,p}\right)\longrightarrow \left(W^{1,q}(\Omega), \tau_{1,q}\right)
		\end{equation*}
		are continuous.
	\end{enumerate}
\end{theorem}
\begin{proof}
	\cite[Theorem 4.20]{AOJA}.
\end{proof}
\begin{definition} \label{defv1p0}
	In the sequel we will write 
	\begin{equation}
		W^{1,p}_{comp}(\Omega):=\{v\in W^{1,p}(\Omega): v\,\text{is compactly supported}\}
	\end{equation}
	and denote the $\rho_{1,p}$-closure of $W^{1,p}_{comp}(\Omega)$, that is $\overline{W_{com}^{1,p}}^{\rho_{1,p}}(\Omega)$, by $U^{1,p}_0(\Omega).$
\end{definition}
\begin{remark}\label{comparisonvu}
	{\normalfont In \cite{AOJA} the authors introduced the space $V^{1,p}_0(\Omega)$ as the modular $\rho_{1,p}$ closure of $C^{\infty}_0(\Omega)$ in $W^{1,p}(\Omega)$. On account of the specific counterexamples in \cite{Harj,Hasto}, in general there holds the strict inclusion
		\begin{equation}
			V^{1,p}_0(\Omega)\subsetneq U^{1,p}_0(\Omega),
		\end{equation}
		even under the assumption $p_+=\sup\limits_{x\in \Omega}p(x)<\infty.$
	}
\end{remark}
\begin{proposition}
	Assume, as in the previous Section that the variable exponent $p\in C(\Omega)$ satisfies $p_- > 1$. Then it holds the inclusion
	\begin{equation*}
		U^{1,p}_0(\Omega)\hookrightarrow W^{1,p_-}_0(\Omega).
	\end{equation*}
\end{proposition}
\begin{proof}
	Obviously, we have the norm inclusion
	\begin{equation*}
		W^{1,p}(\Omega)\hookrightarrow W^{1,p_-}(\Omega)
	\end{equation*}
	and by virtue of Theorem \ref{pqmodularembedding1}, the inclusion $i_{p,p_-}:W^{1,p}(\Omega) \hookrightarrow W^{1,p_-}(\Omega)$ is modularly continuous. Consequently, if $K$ is a norm-closed subset of $W^{1,p_-}(\Omega)$ (or equivalently, $\rho_{1,p_-}$-closed, since $p_-$ is constant)  that contains $W^{1,p_-}_{comp}(\Omega)$, it follows that $i_{p,p_-}^{-1}(K)$ is $\rho_{1,p}$-closed and $W^{1,p}_{comp}(\Omega) \subseteq i_{pp_-}^{-1}(K)$. Since by definition  $U^{1,p}_0(\Omega)\subseteq K\cap W^{1,p}(\Omega)\subseteq K$, it follows that $U^{1,p}_0(\Omega)$ is contained in any norm-closed subset of $W^{1,p_-}(\Omega)$ that contains $W^{1,p}_{comp}(\Omega)$. Therefore since $p_-$ is constant,
	\begin{equation}\label{inclusionv1pw1p-}
		U^{1,p}_0(\Omega)\subseteq W^{1,p_-}_0(\Omega).
	\end{equation}
\end{proof}
\section{The Dirichlet integral on $	U^{1,p}_0(\Omega)$}\label{dirichletintegral}
In this Section it assumed that $\Omega\subset {\mathbb R}^n$ is bounded and that $p:\Omega\rightarrow (1,\infty)$ is continuous and satisfies $p_-=\inf\limits_{x\in \Omega}p(x)>n.$ 

\begin{lemma}\label{mmain}
	Any sequence $(v_j)\subset U^{1,p}_0(\Omega)$, with $(\nabla v_j)$ $\rho_p$-Cauchy in $\left(L^p(\Omega)\right)^n$, must $\rho_p$-converge in $L^p(\Omega)$ to a function $v\in U^{1,p}_0(\Omega).$ Therefore, $(v_j)$ must $\rho_{1,p}$ converge to $v\in U^{1,p}_0(\Omega)$ in $W^{1,p}(\Omega)$.
\end{lemma}

\begin{proof}
	Let the sequence $(v_j)\subset U^{1,p}_0(\Omega)$ be chosen in such a way that $(\nabla v_j)$ is $\rho_p$-Cauchy in $\left(L^p(\Omega)\right)^n$. On account of the modular continuity of the inclusion $L^p(\Omega)\subseteq L^{p_-}(\Omega)$,  $(\nabla v_j)$ is $\rho_{p_-}$-Cauchy in $\left(L^{p_-}(\Omega)\right)^n$ and by virtue of (\ref{inclusionv1pw1p-}), the continuity of the Sobolev embedding
	\begin{equation*}
		W^{1,p_-}(\Omega)\hookrightarrow C(\overline{\Omega})
	\end{equation*}
	and the Poincar\'{e} inequality on $W^{1,p_-}_0(\Omega)$, $(v_j)$ is a Cauchy sequence in $W^{1,p_-}_0(\Omega)$ and for any subindices $i,j$, one has (for some constant $C_i(p_-,\Omega)>0$, $i=1,2$):
	\begin{equation*}
		\|v_j-v_i\|_{C(\overline{\Omega})} \leq C_1(p_-,\Omega) \, \| (v_j-v_i)\|_{W^{1,p_-}(\Omega)} \leq C_2(p_-,\Omega) \, \||\nabla (v_j-v_i)|\|_{L^{p_-}(\Omega)}.
	\end{equation*}
	It follows then that
	\begin{equation*}
		\lim_{i,j \to \infty}\, \rho_p(v_j-v_i) = \lim_{i,j \to \infty}\, \int_{\Omega}\frac{|v_j-v_i|^{p}}{p} \, dx = 0.
	\end{equation*}
	By assumption, one then has
	\begin{equation*}
		\lim_{i,j \to \infty}\, \left(\int_{\Omega} \frac{|v_j-v_i|^p}{p}\, dx + \int_{\Omega} \frac{|\nabla (v_j-v_i)|^p}{p}\, dx\right) =  0,
	\end{equation*}
	and taking into account the fact that $U^{1,p}_0(\Omega)$ is a modularly closed subspace of $W^{1,p}(\Omega)$, it follows that the sequence $(v_j)$ converges modularly in $W^{1,p}(\Omega)$ to a function $v$ in $U^{1,p}_0(\Omega).$  
\end{proof}

\begin{theorem}\label{Dirichletvarphi}
	Let $\Omega \subseteq {\mathbb R}^n$ be a $C^{2}$, bounded domain with boundary $\partial\Omega$, $p\in C({\Omega})$, $p_->n$ (hence the possibility $p_+=\infty$ is admissible). Then for any $\varphi\in W^{1,p}(\Omega)$ such that  $\int\limits_{\Omega}\frac{|\nabla \varphi|^p}{p}dx<\infty$, there exists a unique weak solution $u\in W^{1,p}(\Omega)$, of the Dirichlet problem
	\begin{equation}\label{DiPr}
		\begin{cases}
			\Delta_p u=0,\\
			u|_{\partial \Omega}=\varphi
		\end{cases}
	\end{equation}
	that satisfies the condition
	$\displaystyle \int_{\Omega}(p(x))^{-1}\ |\nabla u(x)|^{p(x)}\, dx < \infty$
	and such that the inequality
	\begin{equation}\label{additionalcondition}
		\int_{\Omega}\ |\nabla u(x)|^{p(x)-2}\ \nabla u(x)\ \nabla (u+v-\varphi)(x)\, dx \leq 0
	\end{equation}
	holds for every $v\in U^{1,p(\cdot)}_0(\Omega)$ such that $\displaystyle \int_{\Omega}(p(x))^{-1}\ |\nabla (v-\varphi)(x)|^{p(x)}\, dx < \infty.$
	The weak formulation of (\ref{DiPr}) reads
	\begin{equation}\label{weakformulation}
		\int\limits_{\Omega}|\nabla u|^{p-2}\nabla u \nabla \phi \, dx=0
	\end{equation}
	for all $\phi\in W_{com}^{1,p}(\Omega)$. The boundary condition is to be interpreted as $u-\varphi\in U^{1,p}_0(\Omega).$
\end{theorem}
\begin{remark}
	{\normalfont It is well known that for any constant exponent $p$ it holds that $U^{1,p}_0(\Omega)=V^{1,p}_0(\Omega)=W^{1,p}_0(\Omega)$, the latter being the norm closure of $C^{\infty}_0(\Omega)$ in $W^{1,p}(\Omega)$. It follows that for constant $p$, the notion of weak solution obtained in Theorem \ref{Dirichletvarphi} coincides with the standard notion of weak solution for the classical constant exponent $p$-Laplacian. Example 3.9 in \cite{Harj} reveals that for $p(x)$ variable, the notion of weak solution of Theorem \ref{Dirichletvarphi} is different from the notion of weak solution in, for example \cite{FZ,MRU}, even under the assumption $p_+<\infty$. }
\end{remark}
\begin{proof}
	Consider the Dirichlet integral, i.e., the functional $\mathcal{F} : U^{1,p}_0(\Omega)\rightarrow [0,\infty]$ defined by
	\begin{equation}\label{Dirichletintegral}
		\mathcal{F}(u) = \int_{\Omega}\frac{|\nabla (u-\varphi)|^p}{p}\, dx.
	\end{equation}
	On account of modular uniform convexity $(UC)$ $\mathcal{F}$ has a unique minimizer in $U^{1,p}_0(\Omega)$.  Indeed, the functional ${\mathcal F}$ is obviously bounded below on $U^{1,p}_0(\Omega)$ and the choice of $\varphi$ guarantees that $\inf\limits_{v\in W_{comp}^{1,p}(\Omega)}{\mathcal F}(v)<\infty$. Take a minimizing sequence $(u_j)\subset W^{1,p}_{comp}(\Omega)$, that is, 
	$$\lim_{j \to \infty}\ {\mathcal F}(u_j) = \inf\limits_{u\in W^{1,p}_{comp}(\Omega)}{\mathcal F}(u) = d.$$
	Exploiting the uniform convexity $(UC)$ of the functional ${\mathcal F}$, it can be shown that $\displaystyle \left(\frac{\nabla u_j}{2}\right)$ is $\varrho_p$-Cauchy in $\left(L^p(\Omega)\right)^n$. Indeed, if this were not the case then one could extract a sequence $(i,j)\rightarrow \infty$ such that 
	\begin{equation}\label{arg1}
		\varrho_p\left(\frac{\nabla u_{j}-\nabla u_i}{2}\right)	\geq \epsilon;
	\end{equation}
	on account ofx (\ref{defunifcom}) this would imply
	\begin{equation}
		d\leq	\varrho_p\left(\frac{\nabla (u_j+u_i)}{2}-\varphi\right)\leq (1-\delta(\epsilon, p_-))\frac{\varrho_p (u_j-\varphi)+\varrho_p(u_i-\varphi)}{2},
	\end{equation}
	which leads to a contradiction by letting $i,j\rightarrow \infty$
	(see \cite[Theorem 7.2]{AOJA} for the details). By virtue of Lemma \ref{mmain}, $\displaystyle \left(\frac{ u_j}{2}\right)$ must $\rho_{1,p}$-converge to a function $u\in W^{1,p}(\Omega)$. Since  $U^{1,p}_0(\Omega)$ is closed in the modular topology of $W^{1,p}(\Omega)$, it follows that $u\in U^{1,p}_0(\Omega)$, which is a vector subspace of $W^{1,p}(\Omega)$. Hence, $2u\in U^{1,p}_0(\Omega)$.  On account of \cite[Theorem 4.22]{AOJA}, it can be assumed, without loss of generality that  $\displaystyle \left(\frac{\nabla u_j}{2}\right)$ converges pointwise $a.e.$ to $\nabla u$. The lower semicontinuity of the Dirichlet integral (\ref{Dirichletintegral}) coupled with Fatou's Lemma then yields
	$$\int_{\Omega}p^{-1}\left|\nabla (2u-\varphi)\right|^pdx\leq \liminf_{j \to \infty} \int_{\Omega}p^{-1}\left |\nabla(u_j-\varphi)\right|^p\, dx.$$
	It is claimed that $2u$ is a minimizer of $\mathcal F$ in $U^{1,p}_0(\Omega)$. As in the previous Section, it is easily derived that
	\begin{align*}
		d& \leq \int_{\Omega}\frac{\left|\nabla\left(\varphi-2u\right)\right|^p}{p}dx \leq \liminf\limits_{k\to \infty}\, \int_{\Omega}\frac{\left|\nabla\left(\varphi -\left(\frac{u_{k}}{2}+u\right)\right)\right|^p}{p}dx \\
		& \leq \liminf\limits_{k \to \infty}\, \liminf\limits_{l \to \infty}\, \int_{\Omega}\, \frac{\left|\nabla \left(\varphi-\left(\frac{u_{k}}{2} + \frac{u_{l}}{2}\right)\right)\right|^p}{p}\, dx\\ 
		&\leq \liminf\limits_{k \to \infty}\, \liminf\limits_{l \to \infty}\,  \frac{1}{2}\left(\int_{\Omega}\frac{\left|\nabla\left(\varphi-u_{k}\right)\right|^p}{p}\, dx+ \int_{\Omega}\frac{\left|\nabla \left(\varphi-u_l\right)\right|^p}{p}\, dx\right)\\
		&=\liminf\limits_{k \to \infty}\, \liminf\limits_{l \to \infty}\, \frac{1}{2}\, \Big(\mathcal{F}(u_k) + \mathcal{F}(u_l)\Big) = d,
	\end{align*}
	i.e., $2u$ is a minimizer of $\mathcal F$ on $U^{1,p}_0(\Omega).$  Also in this case, two arbitrary minimizing sequences must converge to the same limit. The proof is similar to that of \cite[Lemma 6.1]{AOJA} and will be omitted.  The rest of the proof follows 
	from the fact that for any $v\in U^{1,p}_0(\Omega)$ and $\phi \in W^{1,p}_{comp}(\Omega)$ it holds
	\begin{equation}
		\lim_{t\rightarrow 0}\frac{{\mathcal F}(v+t\phi)-{\mathcal F}(v)}{t}=	\int\limits_{\Omega}|\nabla v|^{p-2}\nabla v \nabla \phi\,dx<\infty 
	\end{equation}
	and therefore the right hand side must be zero for when $v=2u$.
\end{proof}
Next, we address the question of uniqueness. To this end, let $v\in U^{1,p}_0(\Omega)$  such that $\displaystyle \int_{\Omega}p^{-1}|\nabla (v-\varphi)|^p\, dx <\infty$ nd $2u$ the minimizer whose existence was proved in the preceding paragraph. Then, it holds
\begin{equation}\label{conditionforuniq}
	\int_{\Omega}|\nabla (2u-\varphi)|^{p-2}\ \nabla (2u-\varphi)\ \nabla (v-2u)\, dx \geq 0.
\end{equation}
Indeed, for any $t \in [0,1]$, convexity yields
\begin{align*}
	\int_{\Omega}\frac{1}{p}\left|\nabla\left(2u+t(v-2u)-\varphi\right)\right|^p\, dx& = \int_{\Omega}\frac{1}{p}\left|\nabla \left((1-t)(2u-\varphi)+t(v-\varphi)\right)\right|^p\, dx\\
	&\leq \int_{\Omega}\frac{1-t}{p}\ \left|\nabla(2u-\varphi)\right|^p\, dx + \int_{\Omega}\frac{t}{p}\ \left|\nabla (v-\varphi)\right|^p\, dx\\ 
	& < \infty.
\end{align*}
Set
$$g_t = \frac{1}{pt}\Big(\left|\nabla\left(2u+t(v-2u)-\varphi\right)\right|^p-|\nabla (2u-\varphi)|^p\Big).$$
It is straightforward to check that 
\begin{equation}\label{limite}
	\lim\limits_{t \to 0^+}\ g_t = |\nabla(2u-\varphi)|^{p-2}\nabla(2u-\varphi))\nabla (v-2u)
\end{equation}
pointwise.  On the other hand for $t \in [0,1]$, convexity yields
\begin{align*}
	g_t&=\frac{1}{pt}\Big(\left|\nabla\left((1-t)(2u-\varphi)+t(v-\varphi)\right)\right|^p-|\nabla (2u-\varphi)|^p\Big)\\ 
	&\leq\ \frac{1}{pt}\Big((1-t)|\nabla (2u-\varphi)|^p+t|\nabla (v-\varphi)|^p-|\nabla (2u-\varphi)|^p\Big)\\ 
	&= \frac{1}{p}\Big(|\nabla (v-\varphi)|^p-|\nabla (2u-\varphi)|^p\Big)=:g,
\end{align*}
which, by assumption, is integrable. It is well known that under these conditions it follows, via the application of Fatou's Lemma to the sequence $g=g_t$, that
\begin{equation}
	\limsup\int_{\Omega}g_tdx\leq \int_{\Omega}\limsup g_tdx.
\end{equation}
Due to the fact that $v-2u\in U^{1,p}_0(\Omega)$, the minimal character of $2u$ yields  $$\int_{\Omega}g_tdx\geq 0$$ for any $t\in {\mathbb R}$. Coupled with (\ref{limite}), this yields $\int_{\Omega}gdx\geq 0$, which proves (\ref{conditionforuniq}).\\
Let $w_i\in W^{1,p}(\Omega)$, $i=1,2$ be solutions of (\ref{DiPr}) satisfying the additional condition (\ref{additionalcondition}).
Then, substituting $v=\varphi-w_2$  and $v= \varphi-w_1$ in (\ref{conditionforuniq}) it follows
\begin{equation}
	\int\limits_{\Omega}|\nabla w_1|^{p-2}\nabla w_1\nabla (w_2-w_1)dx\leq 0\,\,\text{and}\,\,\int\limits_{\Omega}|\nabla w_2|^{p-2}\nabla w_2\nabla (w_2-w_1)dx\geq 0.
\end{equation}
On the other hand, there holds the pointwise equality (see \cite{Lindqvist})
\begin{align}\label{vectineq}
	\left(|\nabla w_1|^{p-2}\nabla w_1-|\nabla w_2|^{p-2}\nabla w_2\right)\left(\nabla w_2-\nabla w_1\right)&=\frac{|\nabla w_1|^{p-2}+|\nabla w_2|^{p-2}}{2}|\nabla (w_1-w_2)|^2\\ \nonumber &+\frac{|\nabla w_1|^{p-2}-|\nabla w_2|^{p-2}}{2}(|\nabla w_1|^2-|\nabla w_2|^2).
\end{align}
The last term is nonnegative; since in our setting $p_->n\geq 2$, the scalar function $x\rightarrow x^{p-2}$ is convex. Using these observations and integrating (\ref{vectineq}) over $\Omega$ one has
\begin{align}
	0&\geq 	\int\limits_{\Omega}\left(|\nabla w_1|^{p-2}\nabla w_1-|\nabla w_2|^{p-2}\nabla w_2\right)\left(\nabla w_2-\nabla w_1\right)\,dx\\  \nonumber & \geq
	\int\limits_{\Omega}2^{2-p}|\nabla (w_1-w_2)|^pdx. 
\end{align}
Thus, $w_1-w_2\in U^{1,p}_0(\Omega)$, is constant in $\Omega$ and vanishes on $\Omega$, hence is must be identically zero (\cite[Lemma 4.24]{AOJA}).

\section{Weak solutions and viscosity solutions}\label{viscsol}
The notion of viscosity solutions is well established. We refer the reader to \cite{CIL} for details. In this Section we simply state the definitions that are relevant to the present note and an existence theorem, whose proof is contained in \cite{CIL}.
\begin{definition}
	Let $F: Y\subset {\mathbb R}^n\times {\mathbb R}\times{\mathbb R}^n\times {\mathcal S}(n)\rightarrow {\mathbb R}$, where ${\mathcal S}(N)$ denotes the set of all $n\times n$ symmetric matrices.
	\begin{equation}\label{F}
		F(x,\nabla u,D^2u)=0 \quad \text{in }\Omega,
	\end{equation}
	with boundary condition $u=\varphi$ on $\partial\Omega$.
	
	A lower semicontinuous function $u$ is a viscosity supersolution if
	$u\ge \varphi$ on $\partial\Omega$ and for every test function $\phi\in C^2(\Omega)$ such that $u-\phi$ has a strict minimum at $x_0\in \Omega$ it holds that
	\[
	F(x_0,\nabla\phi(x_0),D^2\phi(x_0))\ge0.
	\]
	Subsolutions are defined in the analogous obvious way. A function $u$ is said to be a viscosity solution of (\ref{F}) if $u$ is both a viscosity supersolution and a viscosity subsolution.
\end{definition}
\medskip
The next proposition generalizes Proposition 2.3 in \cite{MRU}, since it does not require the variable exponent $q$ to be bounded.
\begin{remark}{\normalfont The main idea of the proof of proposition \ref{weakviscosity} follows the outline of \cite[Proposition 2.3]{MRU}. However, we point out that our definition of the function $\Phi$ below had to be modified to guarantee that $\nabla (\Phi-u)^+\in W^{1,p}_{comp}(\Omega)$ (and hence that it qualifies as an admissible test function in (\ref{contention})). Being zero on the boundary of $B(x_0,r)$ is not a sufficient condition.}
\end{remark}

\begin{proposition}\label{weakviscosity}
	Let $\varphi \in W^{1,q}(\Omega)$, $q\in C^1(\Omega)$, $n<q_-=\inf\limits_{x\in \Omega}q(x)<\infty$ and $u_q$ be the weak solution of problem
	
	\begin{equation}\label{problems}
		\begin{cases}
			-\Delta_{q(x)}u(x)=0, & x\in\Omega\\
			u(x)=\varphi(x), & x\in\partial\Omega
		\end{cases}
	\end{equation}
	given in Theorem \ref{Dirichletvarphi}.
	Then $u_q$ is also a viscosity solution of problem (\ref{problems}).
\end{proposition}

\begin{proof}
	The main idea of the proof follows \cite[Proposition 2.3]{MRU}. However, we point out that our definition of the function $\Phi$ below had to be modified to guarantee that $(\Phi-u)^+\in W^{1,p}_{comp}(\Omega)$ (and hence that it qualifies as an admissible test function in the weak formulation of problem (\ref{problems}), see (\ref{contention}) below). 
	Let $x_0 \in \Omega$ and let $\phi\in C^{2}(\Omega)$ be a test function such that $u(x_0) = \phi(x_0)$ and $u - \phi$ has a strict minimum at $x_0$. We claim that
	\begin{equation}\label{contrad1}
		-\Delta_{p(x_0)}\phi(x_0)\ge 0;
	\end{equation}
	to proceed we recall that (see \cite{MRU}) 
	\begin{align}\label{contrad}
		-\Delta_{p(x_0)}\phi(x_0)
		&=
		-|\nabla \phi(x_0)|^{p(x_0)-2}\Delta \phi(x_0)
		-(p(x_0)-2)|\nabla \phi(x_0)|^{p(x_0)-4}\Delta_\infty \phi(x_0)\\
		\nonumber	& 
		-|\nabla \phi(x_0)|^{p(x_0)-2}\ln(|\nabla \phi|)(x_0)
		\langle \nabla \phi(x_0), \nabla p(x_0) \rangle.
	\end{align}
	
	If (\ref{contrad1}) were not true, there would exist $r>0$ such that $B(x_0,r)\subset \Omega$ and
	
	\begin{equation}
		-\Delta_{p(x)}\phi(x)<0
	\end{equation}
	
	on the open ball $B(x_0,r)$.
	
	Set $m := \inf\limits_{\Omega\setminus B(x_0,r)} (u-\phi)(x)$ and let 
	\begin{equation}\label{Phi}
		\Phi(x):=\phi(x)+\frac{m}{2}. 
	\end{equation}

	Then $\Phi$ satisfies $\Phi(x_0)>u(x_0)$ and the inequality
	
	\begin{equation}\label{6.4}
		-\Delta_{p(x)}\Phi
		=
		-\operatorname{div}(|\nabla \Phi|^{p(x)-2}\nabla \Phi)
		<0
		\quad \text{in } B(x_0,r).
	\end{equation}
	Now, since $(\Phi-u)\in W^{1,p}(\Omega)$ it follows  $(\Phi-u)^+\in W^{1,p}(\Omega)$, since truncations preserve the class $W^{1,p}(\Omega)$. Moreover, a simple computation shows that $(\Phi-u)^+=0$ on $\Omega\setminus B(x_0,r)$ and consequently
	$(\Phi-u)^+\in W_{com}^{1,p}(\Omega)$. Therefore $(\Phi-u)^+$ is an admissible function in the weak formulation (\ref{weakformulation}). Accordingly,
	\begin{equation}\label{contention}
		\int_{\Omega}
		|\nabla u|^{p(x)-2}\nabla u \cdot \nabla(\Phi-u)^+\,dx
		=0.
	\end{equation}
	
	On account of (\ref{6.4}) it follows
	
	\[
	\int_{\Omega}
	|\nabla \Phi|^{p(x)-2}\nabla \Phi \cdot \nabla(\Phi-u)^+\,dx
	<0.
	\]
	
	Therefore, arguing as in (\ref{vectineq}) there holds
	\[
	0\leq	 \int_{\Omega}
	\frac{|\nabla \Phi-\nabla u|^{p}}{2^{p-2}}\,dx
	\]
	\[
	\leq \int_{\Omega}
	\left(
	|\nabla \Phi|^{p(x)-2}\nabla \Phi
	-
	|\nabla u|^{p(x)-2}\nabla u
	\right)
	\cdot \nabla(\Phi-u)\,dx<0.
	\]
	
	This contradiction proves the desired claim.
	
	This proves that $u$ is a viscosity supersolution. The proof that $u$ is a viscosity subsolution runs as above and we omit the details.
\end{proof}
The next theorem is standard and will be necessary for the discussion in the next Section.
\begin{theorem} 
	The problem
	\begin{equation}\label{limitproblemexistence}
		\begin{cases}
			-\Delta_\infty u-|\nabla u|^2\ln|\nabla u|\langle \xi,\nabla u\rangle=0
			& \text{in }\Omega,\\
			u=f & \text{on }\partial\Omega
		\end{cases}
	\end{equation}
	possesses a unique viscosity solution.
\end{theorem}
\begin{proof}
	See \cite{CIL}.
\end{proof}
\section{Main result}\label{main}
We now undertake the final step in the proof of the central contribution in this article.
In the sequel, $(p_j)\subset C^1(\Omega)$ will denote a sequence of admissible exponents such that
\begin{enumerate}
	\item  For all $ j\in {\mathbb N},$ $p_j\in C^1(\Omega)$ \label{c1}  \\
	\item $p_j$ converges to $\infty$ uniformly in $\Omega$, that is, for any $M>0$, there exists $J>0$ such that $p_j(x)>M$ for all $j\geq J, x\in \Omega$ \label{c2}.\\
	\item $\nabla p_j\rightarrow \xi \in C(\Omega)$ uniformly in $\Omega$ \label{c3}.
\end{enumerate}

\begin{remark} {\normalfont We highlight the fact that conditions (1)-(3) do not preclude the possibility $(p_j)_+=\sup\limits_{x\in \Omega}p_j(x)=+\infty$ for all $j$.	This generalizes the treatment in \cite{MRU}. If in condition (3) $\Omega$ is replaced by its closure, $\overline{\Omega}$, then the Harnack inequality proved in \cite{MRU} holds and hence the boundedness of each $p_j$ follows (see \cite[Assertion (2.9)]{MRU}).}
\end{remark}
Next, observe that no generality is lost by assuming 
\begin{equation}\label{pjgreater}
	(p_j)_-=\inf\limits_{x\in \Omega}p_j(x)>\alpha > n
\end{equation}
for all $j$.\\ 

It is well known \cite[Corollary IX.14]{Brezis} that under the condition (\ref{pjgreater}) it holds
\begin{equation}\label{sobolevem}
	W^{1,p_j}(\Omega)\hookrightarrow W^{1,\alpha}(\Omega)\hookrightarrow C^{1-\frac{n}{\alpha}}(\Omega)
\end{equation}
and that there exists a positive constant $C=C(\Omega, \alpha,n)$ such that
\begin{equation}\label{holderestimate}
	\sup\limits_{x\neq y}\frac{|v(x)-v(y)|}{|x-y|^{1-\frac{n}{\alpha}}}\leq C\|v\|_{W^{1,\alpha}(\Omega)}.
\end{equation}
On account of $\||\nabla \varphi|\|_{\infty}\leq 1$, for all $j\in {\mathbb N}$ it holds that
\begin{equation}\label{conditionvarphi}
	\int\limits_{\Omega}\frac{|\nabla \varphi|^{p_j}}{p_j}dx\leq |\Omega|,
\end{equation}
where $|\Omega|<\infty$ is the Lebesgue measure of $\Omega$.
By virtue of Theorem \ref{Dirichletvarphi}, for every natural number $j$ there is a unique minimizer $u_j\in U^{1,p_j}_0(\Omega)$ of the Dirichlet integral (\ref{Dirichletintegral}) with $p=p_j$.\\

\begin{lemma}
	For any $j\in {\mathbb N}$, it holds that
	\begin{equation}
		\int\limits_{\Omega}\frac{|\nabla (u_j-\varphi)|^{p_j}}{p_j}dx\leq |\Omega|.
	\end{equation}
\end{lemma}
\begin{proof}
	The proof is immediate from the observation that $u_j$ is the minimizer on the vector space $U^{1,p_j}_0(\Omega)$ of the functional $F_j$ given by 
	\begin{equation}\label{minim}
		F_j(u)=\int\limits_{\Omega}\frac{|\nabla (u-\varphi)|^{p_j}}{p_j}dx.
	\end{equation}
	The lemma follows from taking $u=0$ in (\ref{minim}) and invoking condition (\ref{conditionvarphi}).
\end{proof}
\begin{lemma}\label{uj/2} For all $j\in {\mathbb N}$, one has:
	\begin{equation}
		\int\limits_{\Omega}\frac{|\nabla (u_j/2)|^{p_j}}{p_j}dx\leq |\Omega|.
	\end{equation}
\end{lemma}
\begin{proof}
	\begin{align}
		\int\limits_{\Omega}\frac{|\nabla (u_j/2)|^{p_j}}{p_j}dx&=
		\int\limits_{\Omega}\frac{|\nabla (u_j-\varphi)/2+\nabla (\varphi/2)|^{p_j}}{p_j}dx\\ &\leq\frac{1}{2}\left( \int\limits_{\Omega}\frac{|\nabla (u_j-\varphi)|^{p_j}}{p_j}+\int\limits_{\Omega}\frac{|\nabla \varphi|^{p_j}}{p_j}dx\right)\\ &\leq |\Omega|.
	\end{align}
\end{proof}
\begin{lemma}\label{normuniformlybounded}
	For any $j\in {\mathbb N}$ the following inequality follows from the preceding lemmas:
	\begin{equation}
		\||\nabla u_j|\|_{p_j}\leq 2e^{e^{-1}} \max\{1,|\Omega|\}.
	\end{equation}
\end{lemma}
\begin{proof} For any measurable exponent $p:\Omega\rightarrow [1,\infty)$
	The modulars $\rho_p(u)=\int\limits_{\Omega}|u|^pdx$ and $\tilde{\rho}_p(u)=\int\limits_{\Omega}\frac{|u|^{p}}{p}dx$ define equivalent Luxemburg norms $\|\cdot\|_{\rho_p}$ and $\|\cdot\|_{\tilde{\rho}_p}$; more precisely
	$\|\cdot\|_{\rho_p}\leq e^{e^{-1}}\|\cdot\|_{\tilde{\rho}_p}$ and $\|\cdot\|_{\tilde{\rho}_p} \leq  \|\cdot\|_{\rho_p}.$ From lemma \ref{uj/2} it follows $\|\frac{|\nabla u_j|}{2}\|_{\tilde{\rho}_{p_j}}\leq \max\{ |\Omega|,1\}$;  the claim can be derived immmediately from this fact.
\end{proof}

\begin{lemma}\label{inclusionofV}
	For every $j\in {\mathbb N}$, $U^{1,p_j}_0(\Omega)\hookrightarrow W^{1,(p_j)_-}_0(\Omega)\hookrightarrow W^{1,\alpha}_0(\Omega)$.
\end{lemma}
\begin{proof}
	The first inclusion is (\ref{inclusionv1pw1p-}), the second one follows from (\ref{sobolevem}).
\end{proof}
\begin{lemma}\label{unifbound}
	The sequence $(u_j)$ is bounded in $W^{1,\alpha}(\Omega)$, that is, there exists a positive constant $C$, independent of $j$ such that
	\begin{equation}
		\|u_j\|_{\alpha}+\||\nabla u_j|\|_{\alpha}\leq C,
	\end{equation}
	for all $j\in {\mathbb N}$.
\end{lemma}
\begin{proof}
	It follows from lemma \ref{inclusionofV}, from Theorem \ref{Dirichletvarphi} and from Poincar\'{e}'s inequality that $\|u_j\|_{\alpha}\leq C \||\nabla u_j|\|_{\alpha}$; notice that $C$ is independent of $j$. The rest of the proof follows on account of Lemma \ref{normuniformlybounded}.
\end{proof}

\begin{lemma}\label{equicontinuity}
	The sequence $(u_j)\subset C^{1-\frac{n}{\alpha}}(\Omega)$ is equicontinuous.
\end{lemma}
\begin{proof}
	The proof follows from Lemma \ref{unifbound} and estimate (\ref{holderestimate}).
\end{proof}

Theorem \ref{theoremain} is the central result of this paper. The fundamental difference between Theorem \ref{theoremain} and Theorem 1.1 in \cite{MRU} is that in contrast to the proof of Theorem 1.1, the proof below covers the possibility $\sup\limits_{x\in \Omega}p_j(x)=\infty$ for every $j$. The inclusion of unbounded exponents in the subsequent argument would not have been possible without the novel mathematical background developed in Sections \ref{modulartopologies} and \ref{dirichletinfinitylaplacian}.
\begin{theorem}\label{theoremain} Fix $\varphi \in W^{1,\infty}(\Omega)$ with $\||\nabla \varphi|\|_{\infty}\leq 1.$ Let $(p_j)$ be a sequence of functions satisfying assumptions $(1)-(3)$. In particular, we allow $\sup\limits_{x\in \Omega}p_j(x)=\infty.$ Assume $p_j(x)\rightarrow \infty$ uniformly in $\Omega$. For each $j$, let $u_j$ be the unique solution to the Dirichlet problem for variable exponent $p_j(\cdot)$-Laplacian
	\begin{equation}\label{problemr}
		\begin{cases}
			\Delta_{p_j(x)}u(x):=\text{div}\big(|\nabla u(x)|^{p_j(x)-2}\nabla u(x)\big) =0, & x\in\Omega\\
			u(x)=\varphi(x), & x\in\partial\Omega,
		\end{cases}
	\end{equation}
	given in Theorem \ref{Dirichletvarphi}.
	Then
	\[
	u_j\to u \quad \text{uniformly in }\Omega,
	\]
	where $u$ is the unique viscosity solution of
	\begin{equation}\label{limitproblem}
		\begin{cases}
			-\Delta_\infty u-|\nabla u|^2\ln|\nabla u|\langle \xi,\nabla u\rangle=0
			& \text{in }\Omega,\\
			u=\varphi & \text{on }\partial\Omega.
		\end{cases}
	\end{equation}
\end{theorem}
\begin{proof}
	
	Owing to Lemma \ref{equicontinuity}, it follows from Ascoli's theorem, extracting a
	subsequence if necessary, that
	\[
	u_j \to u \quad \text{uniformly in } \Omega,
	\]
	for a certain continuous function $u$. Since $u_j = \varphi$ on $\partial\Omega$, it is clear that $u = \varphi$ on $\partial\Omega$.
	
	To prove that $u$ is a viscosity supersolution of problem (\ref{limitproblem}), let $\beta \in C^2(\Omega)$ be such that
	$u-\beta$ has a strict minimum at $x_0 \in \Omega$, with
	$\beta(x_0) = u(x_0)$. We will prove that
	\begin{equation}\label{vissuper}
		-\Delta \beta(x_0) - |\nabla \beta(x_0)|^2
		\ln |\nabla \beta(x_0)| \langle \xi(x_0), \nabla \beta(x_0) \rangle
		\ge 0 .
		\tag{3.1}
	\end{equation}
	
	Since $u_j \to u$ uniformly, there is a sequence $(x_j)_j$ such that
	$x_j \to x_0$ and $u_j - \beta$ has a minimum at $x_j$.
	As $u_j$ is a viscosity solution of problem (\ref{problemr}) it follows that
	\[
	-\frac{|\nabla \beta(x_j)|^2 \Delta\beta(x_j)}{p_j(x_j)-2}
	-\Delta_{\infty}\beta(x_j)
	-|\nabla \beta(x_j)|^2 \ln |\nabla \beta(x_j)|
	\left\langle
	\nabla \beta(x_j),
	\frac{\nabla p_j(x_j)}{p_j(x_j)-2}
	\right\rangle
	\ge 0 .
	\]
	
	Now $x_j \to x_0$, thus, since $2\leq n<\alpha < p_j\rightarrow \infty$ uniformly, 
	\[
	\frac{|\nabla \beta(x_j)|^2 \Delta \beta(x_j)}{p_j(x_j)-2}
	\to 0,
	\]
	
	\[
	\Delta_{\infty} \beta(x_j) \to \Delta_{\infty} \beta(x_0),
	\]
	
	\[
	|\nabla \beta(x_j)|^2 \ln(|\nabla \beta(x_j)|)
	\to
	|\nabla \beta(x_0)|^2 \ln(|\nabla \beta(x_0)|),
	\]
	
	and
	\[
	\left\langle
	\nabla \beta(x_j),
	\frac{\nabla p_j(x_j)}{p_j(x_j)-2}
	\right\rangle
	\to
	\langle \nabla \beta(x_0), \xi(x_0) \rangle .
	\]
	
	Inequality (\ref{vissuper}) follows immediately.
	
	Hence $u$ is a viscosity supersolution; a mutatis-mutandis type of argument proves that $u$ is also a subsolution.
\end{proof}

\end{document}